\documentclass[11pt, twoside]{article}
\usepackage{amssymb, amsmath, amsthm}
\usepackage[top=30mm, left=23.5mm, right=23.5mm, bottom=27mm]{geometry}
\usepackage{fancyhdr}
\usepackage{tikz}
\usetikzlibrary{positioning}
\usepackage{titling}
\usepackage{hyperref}
\predate{}
\postdate{}
\usepackage{lipsum}
\newtheorem{theorem}{Theorem}
\newtheorem{lemma}[theorem]{Lemma}

\newtheorem{conjecture}[theorem]{Conjecture}
\usepackage{titling}
\usepackage{xcolor}
\usepackage{placeins}
\usepackage{tikz}
\usepackage{verbatim}
\usepackage{titling}
\usepackage[T1]{fontenc}
\usepackage{currvita}
\usepackage{bbm}
\usepackage{enumitem}
\newtheorem{corollary}[theorem]{Corollary}

\theoremstyle{definition}

\theoremstyle{remark}

\newtheorem*{theorem*}{Theorem}
\usepackage[T1]{fontenc}
\usepackage[utf8]{inputenc}
\usepackage{lmodern}
\usepackage{listings}
\usepackage[capitalise,noabbrev]{cleveref}
\let\OLDthebibliography\thebibliography
\renewcommand\thebibliography[1]{
  \OLDthebibliography{#1}
  \setlength{\parskip}{0pt}
  \setlength{\itemsep}{0pt plus 0.3ex}
}
\begin{document}
\newcommand{\Addresses}{{
\bigskip
\footnotesize
\medskip

\noindent Maria-Romina~Ivan, \textsc{
Department of Pure Mathematics and Mathematical Statistics, Centre for Mathematical Sciences, Wilberforce Road, Cambridge, CB3 0WB, UK.}\\\nopagebreak\textit{Email address: }\texttt{mri25@cam.ac.uk}

\medskip

\noindent Imre~Leader, \textsc{Department of Pure Mathematics and Mathematical Statistics, Centre for Mathematical Sciences, Wilberforce Road, Cambridge, CB3 0WB, UK.}\\\nopagebreak\textit{Email address: }\texttt{i.leader@dpmms.cam.ac.uk}

\medskip

\noindent Mark~Walters, \textsc{School of Mathematical Sciences, Queen Mary University of London, London, E1 4NS, UK.}\\\nopagebreak\textit{Email address: }\texttt{m.walters@qmul.ac.uk}
\medskip}}

\pagestyle{fancy}
\fancyhf{}
\fancyhead [LE, RO] {\thepage}
\fancyhead [CE] {MARIA-ROMINA IVAN, IMRE LEADER AND MARK WALTERS}
\fancyhead [CO] {GENERALISED PRISMS AND EUCLIDEAN RAMSEY THEORY}
\renewcommand{\headrulewidth}{0pt}
\renewcommand{\l}{\rule{6em}{1pt}\ }
\title{\Large\textbf{GENERALISED PRISMS AND EUCLIDEAN RAMSEY THEORY}}
\author{MARIA-ROMINA IVAN, IMRE LEADER AND MARK WALTERS}
\date{}
\maketitle

\begin{abstract}
A finite subset $X$ of $\mathbb R^d$ is called Ramsey if for every $k$ there exists an $n$ such that whenever $\mathbb R^n$ is $k$-coloured there exists a monochromatic congruent copy of $X$. K\v r\'i\v z showed that if there is a soluble group of symmetries of $X$ that acts transitively on $X$, then $X$ is Ramsey. Determining which sets are Ramsey is a major unsolved problem.

In this paper we show that if there is a finite group of isometries of $\mathbb R^d$ that acts transitively on a set $X$, and also on a set $Y$, then the `prism' formed by $X$ and $Y$ in 
$\mathbb R^{d+1}$ (meaning the set $X$ together with a translate of $Y$ in the direction perpendicular to 
$\mathbb R^d$) is itself contained in a finite set on which a group of isometries acts transitively. Moreover, if the initial group of isometries is soluble then so is the final group. This provides a new tool for generating Ramsey sets.
\end{abstract}

\section{Introduction}

A finite subset $X$ of some Euclidean space $\mathbb R^d$ is called \textit{Euclidean Ramsey} or simplex \textit{Ramsey} if for every $k$ there is an $n$ such that whenever $\mathbb R^n$ is $k$-coloured there is a monochromatic copy of $X$. Here a `copy' means a congruent copy (also called an isometric copy). Erd\H{o}s, Graham, Montgomery, Rothschild, Spencer and Straus \cite{EGMRSS} introduced this concept in the 1970s.  They showed that if a set is not spherical (meaning that it is contained in a sphere) then it cannot be Ramsey. They conjectured that the converse holds, in other words that a set is Ramsey if and only if it is spherical. There is a `rival' conjecture by Leader, Russell and Walters \cite{LRW}, stating that a set is Ramsey if and only if it is a subset of some finite transitive set $Y$ (meaning a set whose symmetry group acts transitively). We stress that $Y$ may live in a higher-dimensional space than $X$ itself. It turns out that these are
distinct conjectures, although this is not immediately obvious (see \cite{LRW}). 

The best positive result to date is by K\v r\'i\v z \cite{K}. He showed that, if a set $X$ has a subgroup $G$ of its symmetry group that is soluble and acts transitively on $X$, then $X$ is Ramsey. Of course, it follows that subsets of such a set are also Ramsey. Often subsets of such a set (called a soluble set) are called \textit{subsoluble}, while subsets of a transitive set are called \textit{subtransitive}.  

There are many examples of Ramsey sets now known. Often these did not originally come from the result of K\v r\'i\v z above, but a remarkable recent paper of Behague \cite{B} shows that, with one possible exception, every
set that is known to be Ramsey is actually subsoluble. (The exception is the 4-dimensional regular polytope known as the 120-cell.) We mention also an attractive prior result of Karamanlis~\cite{KAR},  who showed that every simplex, in any dimension, is subsoluble -- indeed, this is needed for Behague's result. (The fact that simplices are Ramsey was originally proved by Frankl and R\"odl~\cite{FR2}.)

Our aim in this paper is to introduce some new examples of sets that are subsoluble, and hence are Ramsey. Note that in general it may be hard to determine whether or not a set does embed into a soluble set -- indeed, no algorithm is known for answering this (or indeed for determining whether or not a given set is subtransitive).

Our sets are `generalised prisms', defined in the following way. We have two sets $X$ and $Y$ 
in $\mathbb R^d$; a prism on $X$ and $Y$ is then the subset of $\mathbb R^{d+1}$, viewed as $\mathbb R^d \times \mathbb R$, given by $(X \times \{ 0 \}) \cup (Y \times \{ \lambda \})$ for some non-zero $\lambda$. Note that if $X$ and $Y$ coincide then this is just the usual notion of a prism. 

Our main result is as follows. Suppose that the sets $X$ and $Y$ in 
$\mathbb R^d$ are each transitive. Moreover, they are `simultaneously transitive', in the sense that the transitive symmetry groups of $X$ and $Y$ may be taken to be equal. In other words, there exists a finite group $G$ of isometries of $\mathbb R^d$ that acts transitively on $X$ and also on $Y$. 
Then any prism on $X$ and $Y$ is subtransitive. Moreover, if $G$ is soluble then the prism is in fact subsoluble -- in particular, any prism on $X$ and $Y$ is Ramsey.

For example, if $X$ is a rectangle, and $Y$ is a congruent rectangle rotated about its centre by some rational angle, then the prism on $X$ and $Y$ is Ramsey. Perhaps a more interesting consequence is the following. Let $X$ be any subsoluble set in $\mathbb R^d \subset \mathbb R^{d+1}$. Then if we add any one point to $X$, not lying in the hyperplane of $X$, then the resulting set is Ramsey.

The plan of the paper is as follows. We prove the above result in Section 2.
There is also an abstract version of the conjecture that all transitive sets are Ramsey: this is called the block sets conjecture \cite{LRW}. We give the relevant definitions in Section 3, and then go on to
consider the corresponding solubility barrier. In particular, we give an example of a template that has no soluble transitive subgroup of its symmetry group, and yet embeds into a template that does.  
Finally, in Section 4 we present some open problems.

Through the paper our notation is standard. We denote by $D_{2n}$ the dihedral group of size $2n$, i.e. the group of symmetries of a regular $n$-gon, by $C_k$ the cyclic group of order $k$, and by $S_l$ the group of all permutations of $l$ distinct elements. 
\section{New subtransitive sets from old}
The main result of this section is Theorem~\ref{t:prisms}. We show that given two sets with a common symmetry group acting transitively on each on them, `lifting' just one of them above, by any amount, results in a subtransitive set. Furthermore, if the symmetry group is soluble, then the resulting set is also subsoluble, and in particular Ramsey. We make this precise below.

\begin{theorem}\label{t:prisms}
Let $X$ and $Y$ be two finite sets in $\mathbb R^d$, and $G$ a finite group of isometries of $\mathbb R^d$ that acts transitively on both $X$ and $Y$. Then, for any $\lambda\neq 0$ the set $Z=(X,0)\cup (Y,\lambda)\subset \mathbb R^{d+1}$ is a subset of a finite transitive set $W$ in $\mathbb R^m$ for some $m$. Moreover, if $G$ is soluble then the group acting transitively on $W$ may be taken to be soluble as well, so that in particular $Z$ is Ramsey.
\end{theorem}
It is worth mentioning that the result is not true when $\lambda=0$. For example, if $X$ is an equilateral triangle and $Y$ is its centre, then $X\cup Y$ cannot be Ramsey (or subtransitive), as it is not even spherical.

For a very simple example, consider the sets $X=\{-a,+a\}$ and $Y=\{-b,+b\}$ for some positive real numbers $a,b$. Of course, the ambient starting space here is $\mathbb R$, i.e. $d=1$. The reflection in the origin in $\mathbb R$ acts transitively on both $X$ and $Y$, thus the soluble group $C_2$ acts transitively on both $X$ and $Y$. Now, given $\lambda>0$, the set $(X,0) \cup (Y,\lambda)$ is a trapezium. Since any isosceles trapezium can be obtained by choosing $a,b$ and $\lambda$ appropriately, Theorem~\ref{t:prisms} shows that any isosceles trapezium is subsoluble, and consequently Ramsey. 
(The fact that all isosceles trapezia are Ramsey was proved by K\v r\'i\v z \cite{K2}, and the fact that they are subsoluble was proved by Behague \cite{B}.)

As another example, consider $X$ to be an equilateral triangle centred at the origin and $Y$ the equilateral triangle formed by its midpoints. Then $X$, $Y$ and $G=D_6$ satisfy the conditions of Theorem~\ref{t:prisms}, hence, for any $\lambda\neq 0$, the set $(X,0)\cup (Y,\lambda)$ is subsoluble, and consequently Ramsey. The configuration is depicted below, with the final set made up of the red and green points.
\begin{center}
\includegraphics[width=20em]{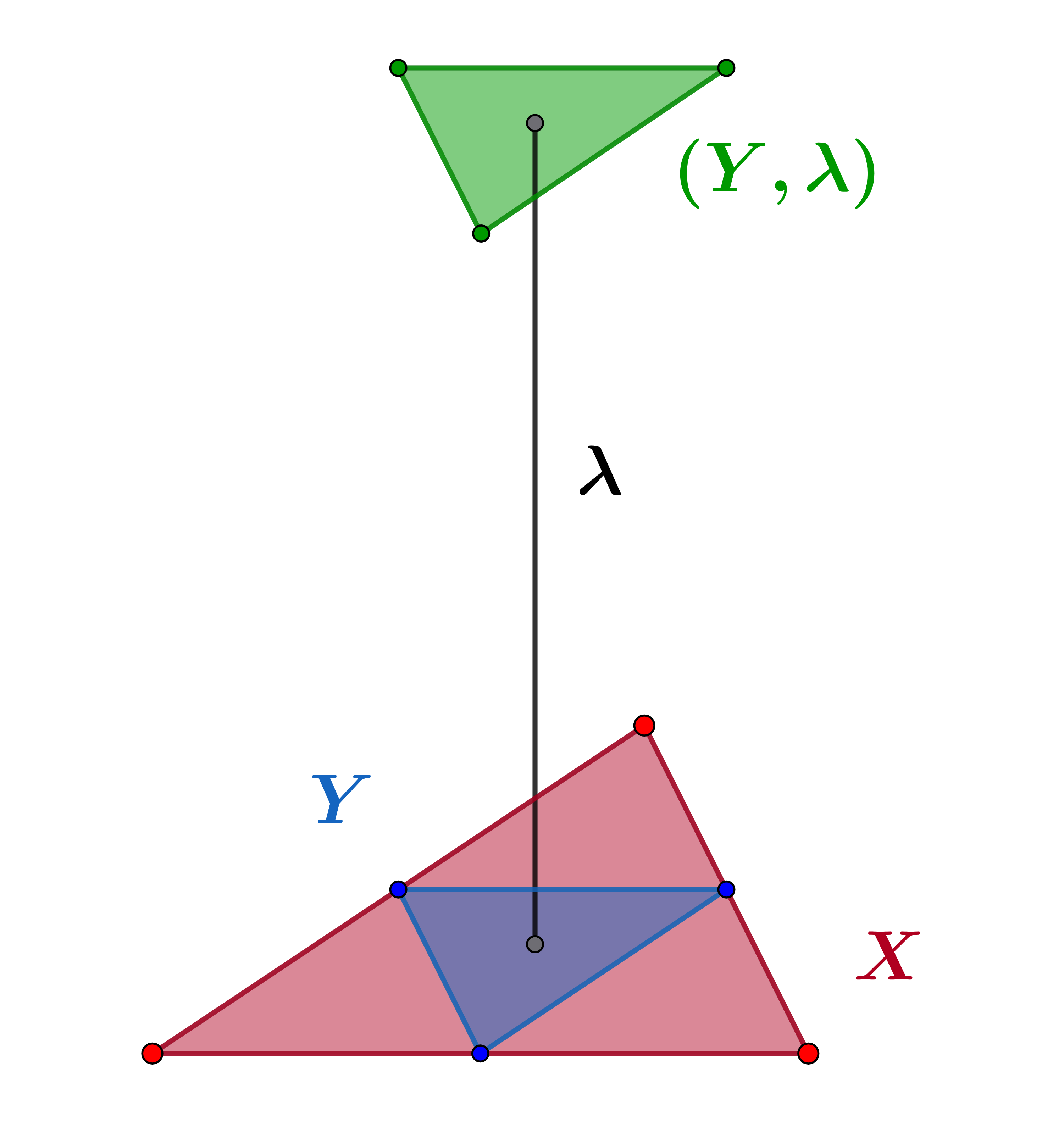}
\end{center}
Our proof consists of two parts. Fix $X$ and $Y$. First, we show in Lemma~\ref{l:deeper} that if the result is true for some $\lambda'>0$, then it is also true for any $\lambda''\geq\lambda'$: this is a standard kind of argument. Secondly, we show in Lemma~\ref{l:higher} that the result is true for a positive sequence $(\lambda_n)_{n\geq1}$ such that $\lambda_n\rightarrow 0$: this is the delicate part of the proof. By reflecting in the plane $x.e_{d+1}=0$ (where $e_{d+1}$ denotes the point $(0,1)$), we also get that the result is true for all negative $\lambda$.

\begin{proof}[Proof of Theorem~\ref{t:prisms}] For the entirety of the proof, $X$ and $Y$ are fixed. In particular, when referring to Theorem~\ref{t:prisms}, we always mean `for the $X$ and $Y$ fixed at the beginning'.
\begin{lemma}\label{l:deeper} Suppose that there exists $\lambda'>0$ for which Theorem~\ref{t:prisms} is true. Then, for any $\lambda''\geq\lambda'$ Theorem~\ref{t:prisms} is also true. In other words, the set $Z'=(X,0)\cup (Y,\lambda)$ is subtransitive, and if $G$ is soluble, then $Z'$ is subsoluble.
\end{lemma}
\begin{proof}
By assumption, $Z=(X,0)\cup(Y,\lambda)$ is subtransitive, therefore it embeds in a finite set $U\subset\mathbb R^m$, for some $m \geq d+1$, such that $U$ has a transitive symmetry group.

Now consider the set $W=U\times \{0,a\}$ for some real number $a\neq 0$. The set $W$ contains an isometric copy of the set $Z\times \{0,a\}$ which, in turn, contains an isometric copy of the set $(X,0,0)\cup (Y,\lambda',a) \subset \mathbb R^{d+2}$.

We choose $a$ such that ${\lambda'}^2+a^2={\lambda''}^2$ (this is possible as $\lambda''\geq\lambda')$, which implies that $(X,0,0)\cup ((Y,\lambda',a)$ is an isometric copy of $Z'$. Therefore, for this choice of $a$, the set $W$ contains an isometric copy of $Z'$.

Since $U$ is transitive, $W$ is transitive with group $C_2\times H$, where $H$ is the symmetry group of $U$. Thus, $Z'$ is subtransitive as claimed. 

Finally, suppose that $G$ is soluble. Then, by assumption, we may assume that $U$ is subsoluble, hence $H$ is soluble. Therefore $H\times C_2$ is soluble, which tells us that $Z'$ is subsoluble, finishing the proof.

For a colour-coded guide to the proof, we provide the picture below.
\begin{center}
\includegraphics[width=35em]{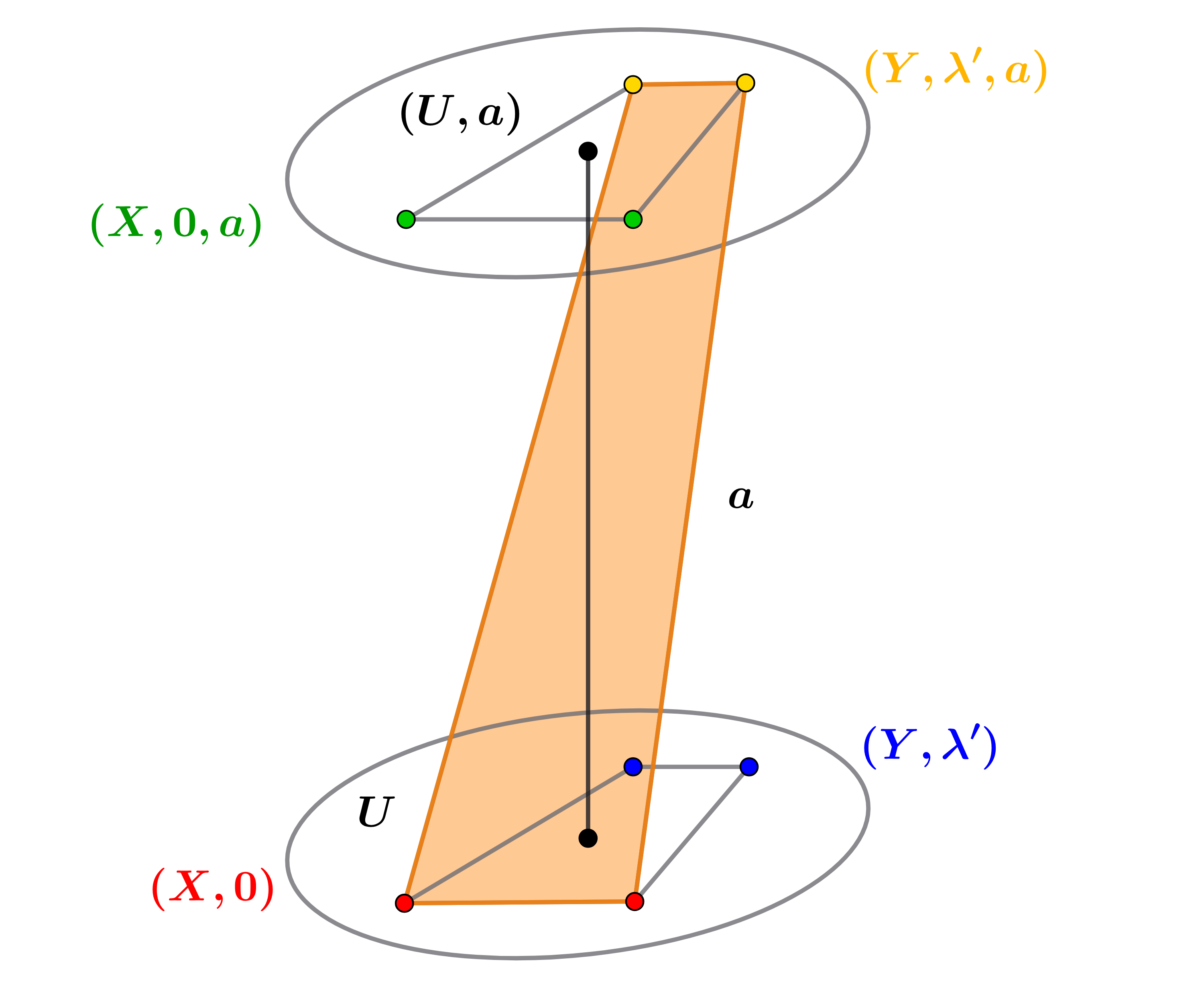}
\end{center}
\end{proof}
It thus suffices to prove Theorem~\ref{t:prisms} for a sequence of arbitrarily small positive values of $\lambda$. 
\begin{lemma}\label{l:higher} Let $x\in X$, $y\in Y$, and $n\in\mathbb N$. Then Theorem~\ref{t:prisms} is true for $\lambda_n=\frac{1}{\sqrt n}\|x-y\|$.
\end{lemma}
\begin{proof} We know that any isometry of $\mathbb R^d$ is an affine map. Since $G$ is finite, by averaging the orbit of a point under the action of $G$, we get that $G$ fixes at least one point, say $v$. By translating the sets $X$ and $Y$ by $-v$, we may assume without loss of generality that $v=0$, hence every element of $G$ is linear.

Let $\alpha>0,\beta>0$ be two real numbers, and consider the set $\{\alpha g(x)+\beta g(y):g\in G\}$. This set is acted on transitively by $G$. Indeed, $h(\alpha g(x)+\beta g(y))=\alpha h\circ g (x)+\beta h\circ g(y)$. 

We will be looking at a sequence of these transitive sets, for carefully chosen $\alpha$ and $\beta$ -- this sequence will `transform', in small steps, the set $X$ into the set $Y$. Finally, taking the union of the products of these cyclically permuted sets will yield our sought-after transitive set, corresponding to $\lambda_n=\frac{1}{\sqrt n}\|x-y\|$.

For all $0\leq i\leq n$ we define the sets $X_i=\{ \frac{n-i}{n} g(x)+\frac{i}{n} g(y):g\in G\}$. As noted above, $G$ acts transitively on every $X_i$, and indeed, we start at $X_0=X$ and finish at $X_n=Y$. Below we exhibit this gradual change for when $X$ and $Y$ are two concentric homothetic equilateral triangles, $x$ is a base point, $y$ is the top point, and $n=3$. This shows that $|X_i|$ can be greater than $\max(|X|,|Y|)$.
\begin{center}
\includegraphics[width=30em]{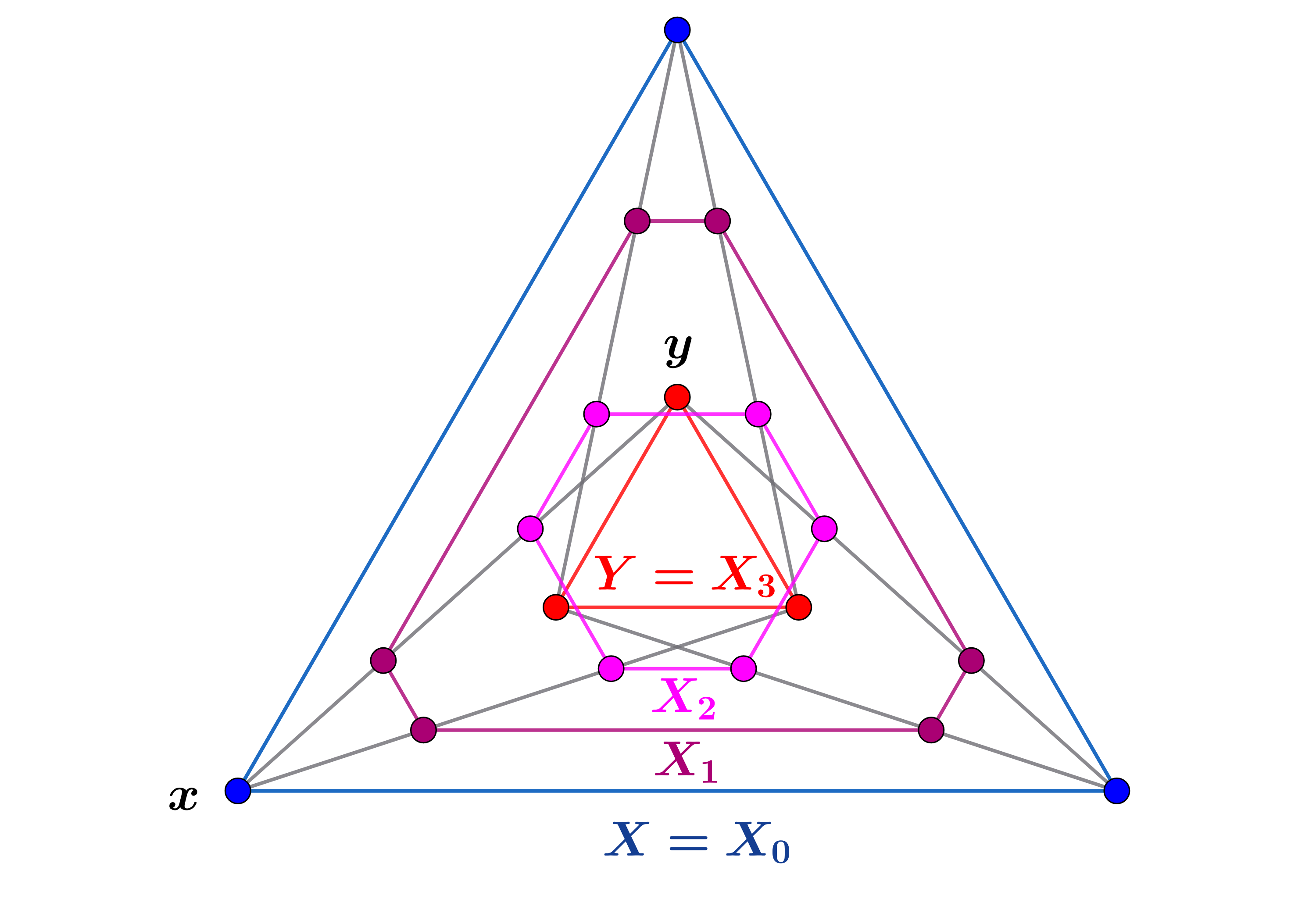}
\end{center}
We now consider the product set $U=X_0\times X_1\times X_2\times\dots\times X_n\subset\mathbb R^{d(n+1)}$. The direct product group of $n+1$ copies of $G$, namely $G^{n+1}$, acts on $U$ pointwise, i.e $(g_0,\dots,g_n).(x_0,\dots,x_n)=(g_0(x_0),\dots,g_n(x_n))$. Since the action of $G$ on each $X_i$ is transitive, $G^{n+1}$ acts transitively on $U$. 

Since we are looking to find parallel copies of $X$ and $Y$, we cycle the components of $U$ in order to bring $X_0=X$ and $X_n=Y$ to the same slot. Finally, to end up with one candidate set, we take the union of all these `rotations'. 

More formally, for all $0\leq i\leq n$ we define $U^{(k)}$ to be the set $X_k\times X_{k+1}\times \dots\times X_n\times X_0\times \dots\times X_{k-1}$. Finally, let $W=\bigcup_{k=0}^nU^{(k)}$, which has a natural transitive action given by the wreath product $G\wr C_{n+1}$. Furthermore, if $G$ is soluble, then so is $G\wr C_{n+1}$, in which case $W$ is soluble.

To be precise, the group $G\wr C_{n+1}$ is the semi-direct product $G^{n+1}\rtimes C_{n+1}$. Let $\sigma$ be the generator of $C_{n+1}$ such that $\sigma(i)=i+1\mod(n+1)$ for all $0\leq i\leq n$. An element of $G\wr C_{n+1}$ is of the form $((g_0,\dots,g_n),\sigma^k)$, where $g_i\in G$ for all $0\leq i\leq n$ and $0\leq k\leq n$. An element $w$ of $W$ is of the form $w=(w_0,\dots, w_n)$ where $w_0\in X_t, w_1\in X_{t+1},\dots,w_n\in X_{t-1}$. The action of $G\wr C_{n+1}$ on $W$ is given by the following:
$$((g_0,\dots,g_n),\sigma^k).w=((g_0,\dots,g_n),\sigma^k).(w_0,\dots,w_n)=(g_0(w_{\sigma^{-k}(0)}),g_1(w_{\sigma^{-k}(1)}),\dots,g_n(w_{\sigma^{-k}(n)})).$$

In order to see that this action is transitive, pick $w$ in $W$, say $w\in U^{(i)}$ for some $0\leq i\leq n$. For a given $0\leq j\leq n$, we want to reach any $w'\in U^{(j)}$. By choosing $k$ appropriately, we first ensure that we land in $U^{(j)}$. Finally, since the action of $G$ is transitive on each $X_l$, we can choose $(g_0,\dots,g_n)\in G^{n+1}$ to reach any element of $U^{(j)}$. Thus the action is indeed transitive on $W$.

To finish the proof, it remains to show that we can find a copy of $(X,0)\cup(Y,\lambda_n)$ in $W$. We start by observing that we have the following copies of $X$ and $Y$ in $W$ (in $U^{(0)}$ and $U^{(n)}$ respectively):
\begin{align*}
X'&=X\times\Big(\frac{n-1}{n}x+\frac{1}{n} y,&&\frac{n-2}{n}x+\frac{2}{n}y,&&\dots&&\frac{1}{n}x+\frac{n-1}{n} y,&&\frac{0}{n}x+\frac{n}{n} y\Big)\\
 Y'&=Y\times \Big(\frac{n}{n}x+\frac{0}{n} y,&&\frac{n-1}{n}x+\frac{1}{n}
  y,&&\dots&&\frac{2}{n}x+\frac{n-2}{n} y,&&\frac{1}{n}x+\frac{n-1}{n} y \Big).
\end{align*}
These sets are translates of $(X,0,\dots,0)$ and $(Y,0,\dots,0)$ respectively, which are isometric copies of $X$ and $Y$. Moreover, both translation vectors are perpendicular to the first copy of $\mathbb R^d$, namely $(\mathbb R^d,0,\dots,0)$. Therefore, there exists some $\beta\geq0$ such that $X'\cup Y'$ is an isometric copy of $(X,0)\cup((Y,\beta)$ (which is an isometric copy of $(X,0)\cup(Y,-\beta)$). Furthermore, $\beta$ must equal the distance between the centres of $X'$ and $Y'$. Since $X$ and $Y$ have the same centre and the difference between the translation vectors is
$\frac1n x-\frac1n y$ in every copy of $\mathbb R^d$, we have that 
$\beta=\sqrt {n\frac{1}{n^2}\|x-y\|^2}=\frac{1}{\sqrt{n}}\|x-y\|=\lambda_n$, which finishes the proof of the lemma. 
\end{proof}
Therefore, putting together Lemma~\ref{l:deeper} and Lemma~\ref{l:higher}, by choosing $n$ sufficiently large, we get that Theorem~\ref{t:prisms} is true for all positive $\lambda$. Finally, as mentioned above, after performing a reflection we get that Theorem~\ref{t:prisms} is true for all $\lambda\neq 0$, which finishes the proof of the theorem.
\end{proof}
The above result provides us with a tool for generating Ramsey sets. The following corollary is an example.
\begin{corollary}\label{cor:piramid}
Let $X$ be a finite transitive set in $\mathbb R^d$ and let $z$ be a point in $\mathbb R^{d+1}$ that does not belong to the hyperplane containing $X$. Then the pyramid with base $X$ and apex $z$ (in other words, the point set $X \cup \{z\}$) is subtransitive. Moreover, if $X$ is subsoluble, then so is the pyramid, which implies that it is Ramsey too.
\end{corollary}
\begin{proof} Let $G$ be a finite transitive group of isometries that acts on $X$. As before, we may assume without loss of generality that $G$ fixes the origin. Let $z=(y,\lambda)$, where $\lambda \neq 0$, and define the  transitive set $Y=\{g(y):g\in G\}$. Then $G$ acts transitively on both $X$ and $Y$, and as such, Theorem~\ref{t:prisms} applies. Therefore, $(X,0)\cup((Y,\lambda)$ is subtransitive, and if $G$ is soluble, then it is in fact subsoluble. However, $(X,0)\cup((Y,\lambda)$ contains $(X,0)\cup\{(y,\lambda)\} $, the desired pyramid, which finishes the proof.
\end{proof}
This tells us, for example, that when $d=2$, any three-dimensional pyramid that has the base a regular polygon is Ramsey. Such an example is illustrated below, where $X$ is a regular hexagon.

\vspace{1em}
\begin{center}
\includegraphics[width=25em]{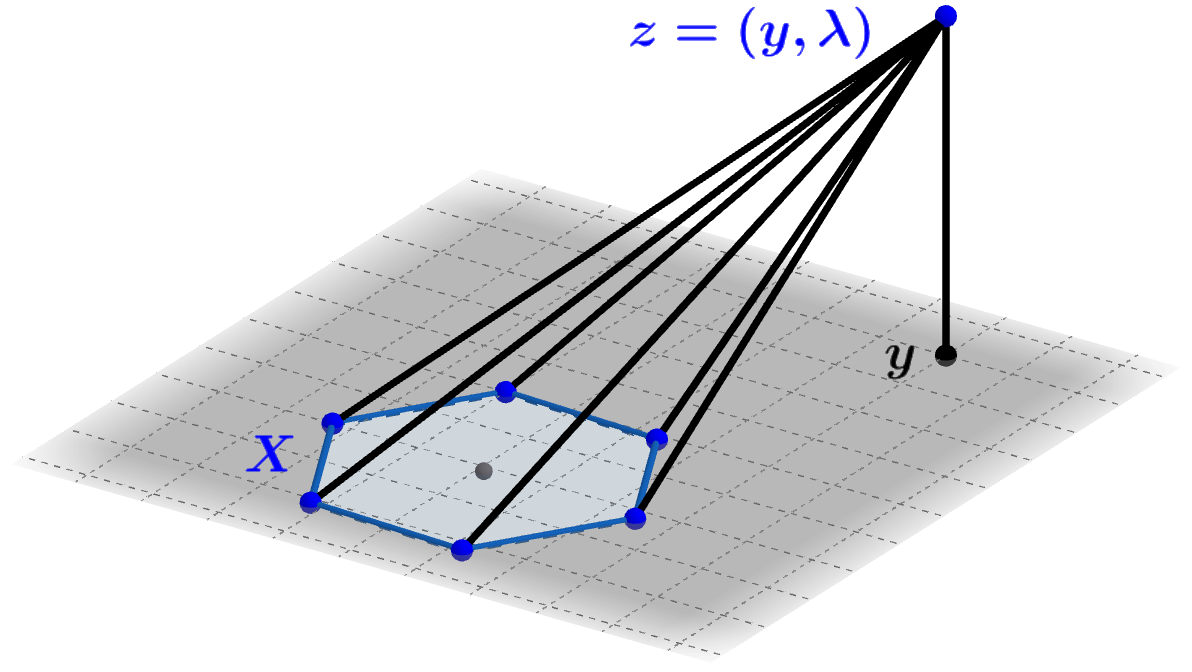}
\end{center}
If `subtransitive' is replaced by `Ramsey', we do not know if Corollary~\ref{cor:piramid} remains true. We discuss this at length in Section 4.

Here is another application. It is important to note that the two regular polygons can have different 
radii. 

\begin{corollary}\label{cor:regular}
Let $X$ and $Y$ be any two regular polygons in $\mathbb R^2$ with the same centre. Then the set $(X,0) \cup (Y,\lambda) \subset \mathbb R^3$ is Ramsey, for any $\lambda \neq 0$.
\end{corollary}

\begin{proof}
    By embedding $X$ and $Y$ into larger regular polygons, we may assume that $m=n$. Then the group of rotations about the centre of $X$ generated by a rotation of angle $2 \pi /n$ acts transitively on each of $X$ and $Y$. The result follows by Theorem~\ref{t:prisms}. 
    \end{proof}

\section{The block sets conjecture and soluble subtransitive sets}
The block sets conjecture, which first appears in \cite{LRW}, is a purely combinatorial statement which implies that every subtransitive set is
Ramsey. In fact, it is equivalent to the assertion that, for
any transitive set $X$, not only is $X$ Ramsey but moreover for any $k$ there is some power $X^n$ such that whenever $X^n$ is $k$-coloured there is a monochromatic
copy of $X$ blown up by a fixed factor $\alpha$ -- in other words, for any $k$ there exist $n$ and $\alpha$ such that $(1/\alpha) X^n$ is `$k$-Ramsey for $X$'.

In order to arrive at the abstract form in which the block sets conjecture is stated, we first introduce some terminology.

We define a \textit{template} over $[m]$ to be a non-decreasing word $T\in[m]^l$ for some $l$. Next, given a template $T$, we define what we mean by a \textit{block set with template $T$}. 

Let $S$ be the set of all words on alphabet $[m]$ of length $l$ that are permutations of the
word $T$.  A \textit{block set} with template $T$ in $[m]^n$ is a set $B$ of words on alphabet $[m]$ of length $n$, formed in the following way. First, select pairwise disjoint subsets $I_1,\dots I_l\subset[n]$ all of the same size $d$ say, and elements $a_i$ for each $i\notin\bigcup_j I_j$. In other words, we select the `block' positions for all letters of $T$, counted with multiplicity, as well as the fixed `reference' word outside them.

And now a word $w$ is in $B$ if and only if $w_i=a_i$ for all $i\notin\cup_j I_j$ and there exists $v\in S$ such that $w_i=v_j$ if $i\in I_j$. We say that this set $B$ is a block set with \textit{block size} or simply \textit{degree} $d$.

We are now ready to state the block sets conjecture.
\begin{conjecture}[\cite{LRW}]
Let $m$ and $k$ be positive integers and let $T$ be a template over $[m]$. Then there exist positive integers $n$ and $d$ such that whenever $[m]^n$ is $k$-coloured there exist a monochromatic block set of degree $d$ with template $T$.
\label{conj:block}
\end{conjecture}
The block sets conjecture has been verified for a few templates, most notably $1\underbrace{2\dots 2}_{s}\underbrace{3\dots 3}_{t}$ for any $s,t\in\mathbb N$ \cite{LRW}, and $1234$ (a direct consequence of K\v r\'i\v z's result and the fact that $S_4$ is soluble). We remark that blocks of size 1 are in general insufficient -- for example for the template 123. (It was shown in \cite{ILW} that in fact blocks of size 2 always suffice for this template, for any number of colours.)

Given a template $T\in[m]^l$, the geometric sets that are associated with it are the following. Let $\alpha_1,\dots\alpha_n \in\mathbb R$. Then the set of all points in $\mathbb R^l$ whose coordinates have as many $\alpha_i$ as $T$ has $i$, for all $i\in[m]$, corresponds to the template $T$. In other words, if the block sets conjecture is true for $T$, then this set is Ramsey. Moreover, if $\alpha_1,\dots,\alpha_n$ are algebraically independent, then the only symmetries of $X$ are the ones given by permuting the $l$ coordinates. 

Therefore, we define the \textit{symmetry group of a template} $T\in[m]^l$ to be $S_l$, which acts on the set of all the permutations of $T$ (by permuting the coordinates). Thus, in what follows, a subgroup of the symmetry group of a template $T\in[m]^l$ is equivalent to a subgroup of $S_l$ (with the same action). 

Note that if a template $T$ has a soluble group of symmetries that acts transitively then the corresponding geometric set $X$ (for any choice of the $\alpha_i$) has the property, by 
K\v r\'i\v z's result, that, for every $k$, there is some blowup of $X^n$ for some $n$, that is
$k$-Ramsey for $X$. And then it follows, by considering the case when the $\alpha_i$ are algebraically independent, that the block sets conjecture holds for that template. 

Now, for several templates $T$ that do \textit{not} have such a subgroup of symmetries, the result of Behague shows that the corresponding geometric sets do embed into soluble sets. Hence those geometric sets are Ramsey, but now we unfortunately  cannot apply the `backwards' direction and deduce that the template satisfies the block sets conjecture. So it is natural to ask if there can ever be a template $T$ that does not have a soluble transitive subgroup of its symmetries, but that embeds into a template that does have one -- it would then follow that $T$ satisfies the block sets conjecture.

Our aim now is to show that this can indeed happen. This is perhaps quite surprising; we do not know if this is an isolated case or the start of a more general pattern.

\begin{theorem}
\label{thm:bsc}
The template $1223333$ does not have a transitive soluble symmetry group, but it embeds in the template $12233333$, which has a transitive soluble symmetry group. \end{theorem}
\begin{proof}
We first note that if a subgroup of $S_7$ acts transitively on the set of all permutations of $1223333$, then, by tracking the position of `1', we see that it must be a transitive subgroup of $S_7$.

To show that $1223333$ does not have a transitive soluble symmetry group, we show that $S_7$ does not have a soluble subgroup that acts transitively on the set of all permutations of $1223333$. There are $\frac{7!}{2!4!}=105$ permutations of $1223333$, thus if a group $G$ acts transitively on them it must have order at least $105$ (fix $x$ a permutation; then the set $\{g.x:g\in G\}$ must have all $105$ permutations).

However, it is well known that the only transitive subgroups of $S_7$, up to conjugacy, of at least this size are $S_7,A_7$, and $L(3,2)$ \cite[Appendix A]{conway}, neither of which is soluble (see for example \cite{groups}). This finishes the first claim, that $1223333$ does not have a transitive soluble symmetry group.
     
Therefore, it remains to show that $12233333$ has a transitive soluble symmetry group. As before, we are searching for a group that acts transitively on a set of size $\frac{8!}{2!5!}=168$, therefore its order must be at least $168$.

Let us first look at the finite field $\mathbb F_8$. To avoid confusion, we write $0_F$ and  $1_F$ be the additive and multiplicative identity respectively. Moreover, we know that $\mathbb{F}_8\setminus\{0_F\}$ equipped with multiplication is a cyclic of order 7. Let $\alpha$ be a generator. Furthermore, let $G$ be the group of affine semilinear transformation on $\mathbb{F}_8$, denoted in literature by $A\Gamma L_1(\mathbb F_8)$. This group is comprised of actions on $\mathbb F_8$ of the form $x\mapsto ax^\sigma+b$ where $a,b\in F_8$, where $a\not=0$ and $\sigma\in\{1,2,4\}$. This group has size $7\times3\times8=168$. Moreover, it is a standard fact that the group $A\Gamma L_1(\mathbb F_q)$ soluble for all prime powers $q$.

Let $X$ be the set of all permutations of $12233333$. We identify the positions with the elements of $\mathbb F_8$ in the following order $0_F,1_F,\alpha,\alpha^2,\dots,\alpha^6$. For example, in the permutation $12323333$, `1' is at $0_F$ and the `2's are at $1_F$ and $\alpha^2$. Therefore, $G$ acts on an element of $X$ by acting on the positions, identified as elements of $\mathbb F_8$. For example, the map $\alpha x$ sends $12233333$ to $13223333$ since it fixes $0_F$, thus `1' and cycles the rest of coordinates, as shown below.
\begin{alignat*}{18}
12233333\hspace{1em} &&\overset{1}{\underset{0_F}{\rule{1em}{1pt}}.} && \overset{2}{\underset{1_F}{\rule{1em}{1pt}}}. && \overset{2}{\underset{\alpha}{\rule{1em}{1pt}}}.&& \overset{3}{\underset{\alpha^2}{\rule{1em}{1pt}}}. && \overset{3}{\underset{\alpha^3}{\rule{1em}{1pt}}}.&&\overset{3}{\underset{\alpha^4}{\rule{1em}{1pt}}}. && \overset{3}{\underset{\alpha^5}{\rule{1em}{1pt}}}.\overset{3}{\underset{\alpha^5}{\rule{1em}{1pt}}}\\
\alpha.12233333\hspace{1em} &&\overset{1}{\underset{0_F}{\rule{1em}{1pt}}.} && \overset{2}{\underset{\alpha}{\rule{1em}{1pt}}}. && \overset{2}{\underset{\alpha^2}{\rule{1em}{1pt}}}.&& \overset{3}{\underset{\alpha^3}{\rule{1em}{1pt}}}. && \overset{3}{\underset{\alpha^4}{\rule{1em}{1pt}}}.&&\overset{3}{\underset{\alpha^5}{\rule{1em}{1pt}}}. && \overset{3}{\underset{\alpha^6}{\rule{1em}{1pt}}}.\overset{3}{\underset{1_F}{\rule{1em}{1pt}}}&&\hspace{0.2em}=\hspace{0.2em}&&\overset{1}{\underset{0_F}{\rule{1em}{1pt}}.} && \overset{3}{\underset{1_F}{\rule{1em}{1pt}}}. && \overset{2}{\underset{\alpha}{\rule{1em}{1pt}}}.&& \overset{2}{\underset{\alpha^2}{\rule{1em}{1pt}}}. && \overset{3}{\underset{\alpha^3}{\rule{1em}{1pt}}}.&&\overset{3}{\underset{\alpha^4}{\rule{1em}{1pt}}}. && \overset{3}{\underset{\alpha^5}{\rule{1em}{1pt}}}.\overset{3}{\underset{\alpha^5}{\rule{1em}{1pt}}}
\end{alignat*}
To show that this action is transitive on $X$, let $\bf{x}$ be an arbitrary permutation of $12233333$. We will show that we can map $\bf{x}$ by repeatedly acting on it with elements of $G$ so that the 1 is at $0_F$, and the 2's are at $\{1_F,\alpha\}$, i.e. to $12233333$. First, let $b\in\mathbb F_8$ be the position of `1' in $\bf{x}$. By applying the map $g_1(x)=x-b$ we send the `1' to $0_F$. Next, let $a\in\mathbb F_8\setminus\{0_F\}$ be the position of one of the `2's in $g_1\bf{x}$. By applying the map $g_2(x)=a^{-1}x$ we send one of the `2's to $1_F$, while keeping the `1' at $0_F$. The other 2 is now at some element $c\in \mathbb F_8\setminus\{0_F\}$. Since $\alpha$ is a generator of $\mathbb F_8^\setminus\{0_F\}$ we must have that $c=\alpha^k$ for some $1\leq k<7$. If $k$ is one of $1,2,4$, then by letting $\sigma$ equal $1,4,2$ respectively and applying the map $x\mapsto x^\sigma$ (which fixes $0_F$ and $1_F$) we obtain the desired permutation. If $k\in\{3,5,6\}$ then, by applying the map $x\mapsto c^{-1}x$ we fix the `1', send one `2' to $1_F$ and the second one to $c^{-1}$. Since $c^{-1}=\alpha^{k'}$ where $k'\in \{1,2,4\}$, we are back in the previous case, which finishes the proof.
\end{proof}
\section{Open Problems}

Perhaps the most intriguing open problem is to determine whether or not Corollary~\ref{cor:piramid} extends to the setting of general Ramsey sets. In other words, if we assume merely that $X$ is Ramsey 
(whether or not it is subtransitive), does it follow that $X \cup \{ z \}$ is Ramsey?

\begin{conjecture}\label{conj:onepoint}
Let $X$ be a Ramsey set in $\mathbb R^d$ and let $z$ be a point in $\mathbb R^{d+1}$ that does not belong to the hyperplane containing $X$. Then the set $X \cup \{ z \}$ is Ramsey.
\label{conj:onepoint}
\end{conjecture}
What is intriguing about this conjecture is that it would follow from \textit{either} of the two competing conjectures about Ramsey sets: that they are the spherical sets or that they are the subtransitive sets. Indeed, apart from statements of the form `the transitive set $S$ is Ramsey' (which of course would follow from either conjecture), we do not know of any other conjectured statements with this
property.

We mention that if $z$ is far enough away from $X$ then $X \cup \{ z \}$ must at least be 2-Ramsey.
To see this, suppose that the perpendicular distance $t$ from $z$ to the plane of $X$ is greater than the
circumradius of $X$. Given a red-blue colouring of $\mathbb R^{2d+1}$, we may find a monochromatic copy of $X$: say that this copy is red. But then, unless we have a monochromatic copy of $X \cup \{ z \}$, it follows that there is a sphere of radius $t$ in $\mathbb R^{d+1}$ that is entirely blue. Now, each copy of $X$ living on this sphere forces two points to be red (the two possible corresponding positions of $z$), and putting these together as the copy of $X$ varies we obtain a larger sphere that is entirely blue. If we keep going, we eventually obtain a monochromatic sphere that is large enough to contain a copy of $X \cup \{ z \}$, as required.

There are some very natural prism-type configurations that are not covered by our main result. One intriguing example is as follows: the set in $\mathbb R^3$ consisting of a rectangle and, above it, the same rectangle rotated by an irrational multiple of $\pi$ (about its centre), as illustrated below.
\begin{center}
\includegraphics[width=27em]{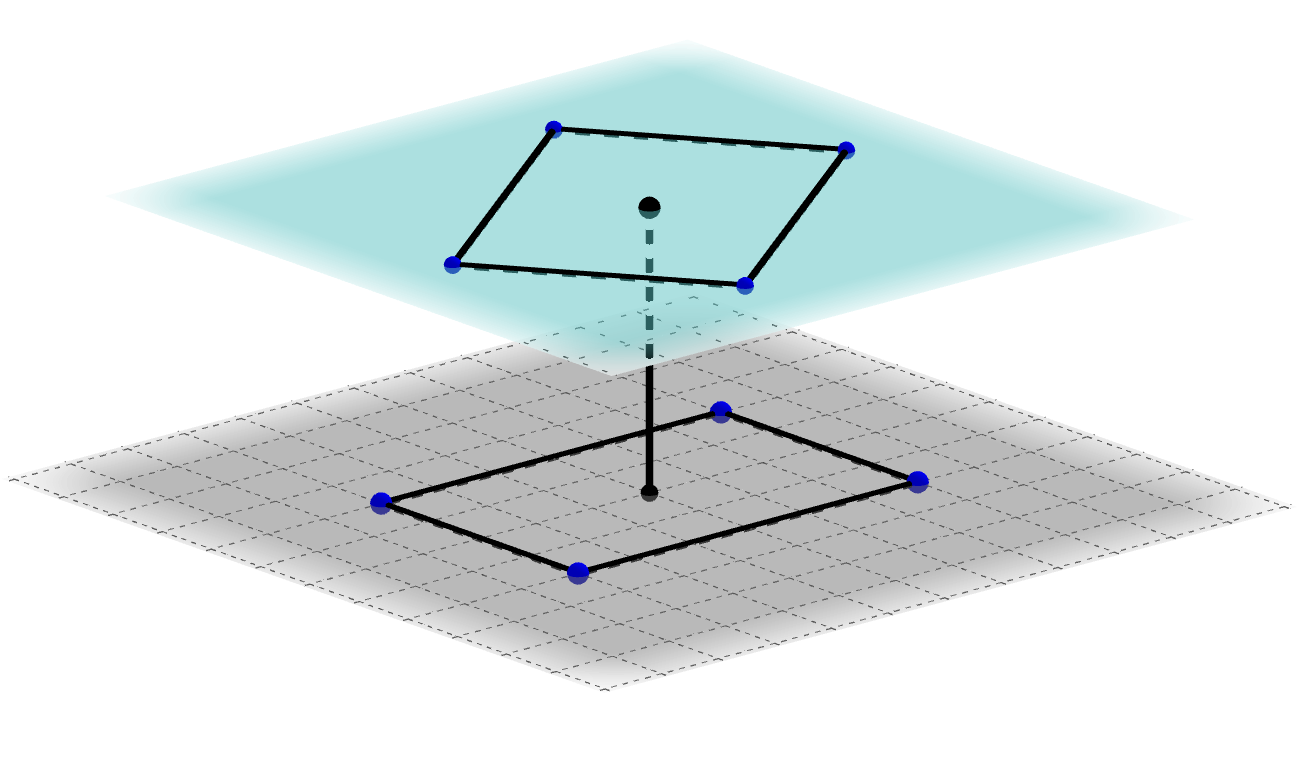}
\end{center}
The point is that, while each rectangle is certainly transitive, the actual group of isometries that acts transitively is not the same in the two cases -- the two groups are isomorphic, but they are not equal. Another example, also in $\mathbb R^3$, consists of a rectangle together with, above it, two points at the same height whose midpoint does not lie above the centre of the rectangle but with the vector from the midpoint to the centre of the rectangle being perpendicular to the line segment (this is to ensure that the set is spherical), and with the angle between the line segment and the rectangle not being a rational multiple of $\pi$. See the picture below. It would be very interesting to know if these configurations are Ramsey.
\begin{center}
\includegraphics[width=27em]{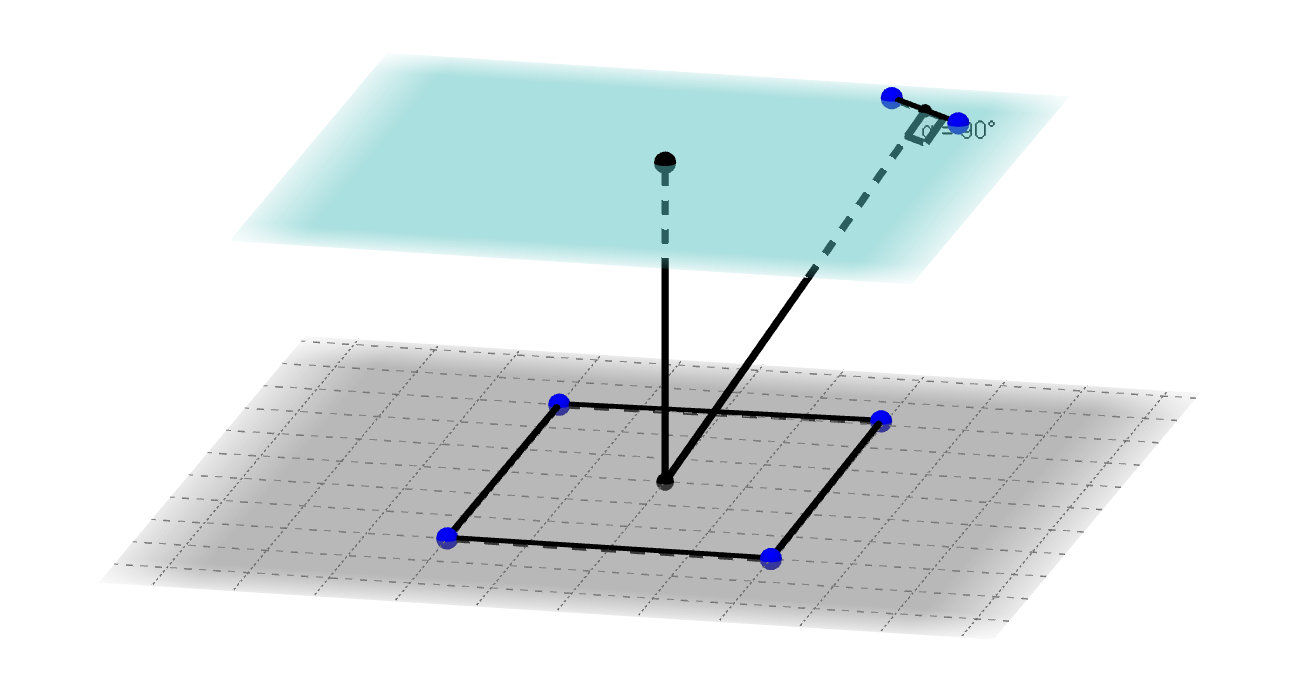}
\end{center}

\bibliographystyle{amsplain}
\bibliography{bibliography}
\Addresses
\end{document}